\documentclass[a4paper,11pt]{article}
\usepackage{xcolor}   
\usepackage{amsmath}  
\usepackage{amssymb}   
\usepackage{amsfonts} 
\usepackage{mathrsfs} 
\usepackage{amsthm}  
\usepackage{indentfirst}
\usepackage[title]{appendix}
\usepackage{hyperref} 
\hypersetup{hypertex=true,
	colorlinks=true,
	linkcolor=blue,
	anchorcolor=blue,
	citecolor=blue}
\usepackage{microtype} 
\newtheorem{theorem}{Theorem}[section]

\newtheorem{proposition}[theorem]{Proposition}
\newtheorem{lemma}[theorem]{Lemma}
\theoremstyle{condition}

\theoremstyle{remark}
\newtheorem{remark}{Remark}
\numberwithin{equation}{section}
\allowdisplaybreaks[4] 

\begin{document}
	\pagenumbering{arabic}

\bigskip\bigskip
\noindent{\Large\bf Precise large deviations for the total population of heavy-tailed critical branching processes with immigration
		\footnote{ This work was supported in part by NSFC (NO. 11971062) and  the National Key Research and Development Program of China (No. 2020YFA0712900).} }
	
\noindent
{Jiayan Guo$^{*}$\footnote{ *corresponding author, Email: guojiayan@mail.bnu.edu.cn. School of Mathematical Sciences \& Laboratory of Mathematics and Complex Systems, Beijing Normal University, Beijing 100875, P.R. China.}
		\quad 	
Wenming Hong\footnote{ Email: wmhong@bnu.edu.cn. School of Mathematical Sciences \& Laboratory of Mathematics and Complex Systems, Beijing Normal	University, Beijing 100875, P.R. China.} }

\begin{center}
\begin{minipage}{12cm}
\begin{center}\textbf{Abstract}\end{center}
\footnotesize
			
We focus on the partial sum $S_{n}=X_{1}+\cdots+X_{n}$ of the critical branching process with immigration $\{X_{n}\}$, when the offspring $\xi$ is regularly varying with index $\nu+1$ and the immigration $\eta$ is regularly varying with index $\delta$ $(0\leq \nu<\delta<1)$. The precise large deviation probabilities for $S_{n}$ are specified, that is, for some appropriate sequences $\{x_{n}\}$ and $\{y_{n}\}$, uniformly for $x_{n}\leq x\leq y_{n}$, $P(S_{n}>x)\sim nx^{-\delta/(1+\nu)}L(x)$, where $L(x)$ is a slowly varying function. Different from that of the subcritical case, here the upper bound $y_n$ is needed. Essentially, this is because the tail probability of the stationary distribution is determined by the offspring or the immigration in the subcritical case. But it is determined by both when the process is critical.
			
\bigskip


\end{minipage}
\end{center}

\section{Introduction}
	
	Let $X_{0}=0$ and  
	\begin{equation}\label{defGWI1}
		X_{n}=\sum_{i=1}^{X_{n-1}}\xi_{n,i}+\eta_{n}, \quad n\in\mathbb{N_{+}}, 
	\end{equation}
	(with the convention $\sum_{i=1}^{0} =0$), where $\{\xi_{n,i}\}_{n,i\in\mathbb{N_{+}}}$ and $\{\eta_{n}\}_{n\in\mathbb{N_{+}}}$ are two independent i.i.d. sequences of nonnegative integer-valued random variables.  Use $\xi$, $\eta$ for the generic copies and to exclude trivialities, we always assume that $P(\eta=0)<1$. We call $\{X_{n}\}$  a {\it  branching process with immigration with offspring $\xi$ and immigration $\eta$}.  Write $\alpha:=E\xi$, $\beta:=E\eta$ for their means, respectively. When $\alpha:=E\xi<1(=1, >1)$, we say the process is subcritical (critical, supercritical). 
	
	We are interested in the large deviation  probabilities $P(S_{n}>x)$ for the partial sum of the process, where 
		\[S_{n}=X_{1}+\cdots+X_{n}.\]	
	
	For the subcritical case, when Cram\'{e}r's condition is satisfied, namely, for some $\theta>0$, $Ee^{\theta\xi}<\infty$ and $Ee^{\theta\eta}<\infty$, Yu et al.  \cite{Yushihang} have provided the exact form of large and moderate deviations for the empirical mean of population ${S_{n}}/{n}$ and centered total population $S_{n}-ES_{n}$, where the rate functions are explicitly identified. When Cram\'{e}r's condition is not satisfied, Guo and Hong \cite{Guo1} proved a  precise behavior of large deviation probability $P(S_{n}>x)$, under the condition that the distribution of $\xi$ or $\eta$ is heavy-tailed. Then for some appropriate sequence $\{x_{n}\}$, uniformly for $x\geq x_{n}$,  $P(S_{n}>x)\sim c_{1}nP(\eta>x)$ or $P(S_{n}>x)\sim c_{2}nP(\xi>x)$, where $c_{1}$ and $c_{2}$ are constants related to the mean of offspring and immigration. Similar large deviation probability for subcritical branching process in random environment  was given in Guo et al. \cite{Guo2}.

	In the present paper,  we are interested in the critical case. Throughout this paper we assume that the generating functions of $\xi$ and $\eta$ have the following representation:
	\begin{equation}\label{conditionf}
		\quad f(x)=x+(1-x)^{1+\nu}L_{1}\left((1-x)^{-1}\right),
	\end{equation}
	\begin{equation}\label{conditiong}
		g(x)=1-(1-x)^{\delta}L_{2}\left((1-x)^{-1}\right),
	\end{equation}
	where $0\leq\nu<\delta<1$ and $L_{i}$ $(i=1,2)$ are slowly varying functions at infinity, i.e., for $\forall\lambda>0$, $L_{i}(\lambda x)/L_{i}(x)\rightarrow1$ as $x\rightarrow\infty$.

	It is shown in Foster and Williamson \cite{Foster} that if $\alpha\leq 1$, $\{X_{n}\}$ converges in distribution to a proper limit $X$  if and only if
	\begin{equation}\label{condition}
		\int_{0}^{1}\frac{1-g(x)}{f(x)-x}dx<\infty.
	\end{equation} 
	Note that condition \eqref{condition} is fulfilled under \eqref{conditionf} and \eqref{conditiong}, then as a consequence, there exists a  stationary distribution $X$ for the sequence  $\{X_{n}\}_{n\in\mathbb{N}}$.	We prove that $X$ is regularly varying with index $\delta-\nu$, and thus obtain precise large deviation probabilities of partial sum $S_{n}$. Since the conclusions may vary when $\nu$ is different, we divide them into two subsections.
	
	\emph{Throughout this paper, $f(x)=o(g(x))$ means $\lim\limits_{x\rightarrow\infty}f(x)/g(x)=0$, and $f(x)\sim g(x)$ means $\lim\limits_{x\rightarrow\infty}f(x)/g(x)=1$, for two vanishing (at infinity) functions. When the value of a positive constant is not of interest, we write $c$ for it.}

\subsection{$\nu\in(0,1)$: heavy-tailed case}
When $\nu\in(0,1)$, since $\forall x\in[0,1)$,
	\[\quad\quad\frac{f(x)-x}{(1-x)^{2}}=\sum_{n=0}^{\infty}\sum_{k=n+1}^{\infty}P(\xi>k)x^{n}=(1-x)^{\nu-1}L_{1}\left((1-x)^{-1}\right),\]
	\[\frac{1-g(x)}{1-x}=\sum_{n=0}^{\infty}P(\eta>n)x^{n}=(1-x)^{\delta-1}L_{2}\left((1-x)^{-1}\right),\]
	we have, by Monotone Density Theorem and Tauberian Theorem (for example, see Theorem 1.7.2 and Corollary 1.7.3 of Bingham et al. \cite{Bingham}), both the offspring $\xi$ and the immigration $\eta$ are regularly varying, i.e., as $x\rightarrow\infty$, 
	\begin{equation*}
		P(\xi>x)\sim cx^{-(\nu+1)}L_{1}(x),
	\end{equation*}
	\begin{equation*}
		P(\eta>x)\sim cx^{-\delta}L_{2}(x),
	\end{equation*}	
	so it is called the {\it heavy-tailed case}.	
	
		Since $L_{1}(x)$ is a slowly varying function, for $\forall\lambda$, we can write
	\[\frac{L_{1}(\lambda x)}{L_{1}(x)}=1+\Gamma_{1}(\lambda,x),\]
	where $\Gamma_{1}(\lambda,x)\rightarrow0$ as $x\rightarrow\infty$. When some positive function $\alpha(x)$ is given so that $\alpha(x)\rightarrow0$ and $\Gamma_{1}(\lambda,x)=o(\alpha(x))$ as $x\rightarrow\infty$ for $\forall \lambda>1$, $L_{1}(x)$ is called slowly varying function with remainder. 

	Hence  we suppose that for $\forall \lambda>1$, 
	\begin{equation}\label{conditionA1}\tag{A1}
	\Gamma_{1}(\lambda,x)=o\left(\frac{L_{1}(x)}{x^{\nu}}\right)\quad \text{as} \; x \rightarrow\infty,
	\end{equation} 
	which is equivalent to
	$L_{1}(x)=C_{1}+o(x^{-\nu})$ as $x \rightarrow\infty,$
	where $C_{1}$ is a positive constant (see Corollary 3.12.3 of Bingham et al. \cite{Bingham} or Proposition 2 of Imomov and Tukhtaev \cite{Imomov}).
	
	Similarly, for
	\[\frac{L_{2}(\lambda x)}{L_{2}(x)}=1+\Gamma_{2}(\lambda,x),\]
	we suppose that $\forall \lambda>1$, 
	\begin{equation}\label{conditionA2}\tag{A2}
		\Gamma_{2}(\lambda, x)=o\left(\frac{L_{2}(x)}{x^{\delta}}\right)\quad \text{as} \; x \rightarrow\infty.
	\end{equation} 
	and it is equivalent to
	$L_{2}(x)=C_{2}+o(x^{-\delta})$ as $x \rightarrow\infty,$
	where $C_{2}$ is a positive constant.
	
	The tail behavior of the stationary distribution $X$ in heavy-tailed case is contained in Proposition 1 of \cite{Imomov}, after applying the Tauberian Theorem.
	
\begin{theorem}[\cite{Imomov}]\label{theoremX1}
		If $0<\nu<\delta<1$ and \eqref{conditionA1}-\eqref{conditionA2} are satisfied, then as $x\rightarrow\infty$,
		\[P(X>x)\sim C(\Gamma(1+\nu-\delta))^{-1}x^{-(\delta-\nu)},\]
		where $C:=C_{2}C_{1}^{-1}(\delta-\nu)^{-1}$ is a positive constant, and $\Gamma(\cdot)$ is the Gamma function.
	\end{theorem}

	By using the branching properties and the tail probability of $X$, we can obtain precise large deviation probability of partial sum $S_{n}$. We have the following,
	
	\begin{theorem}\label{theorem1}
	If $0<\nu<\delta<1$ and \eqref{conditionA1}-\eqref{conditionA2} are satisfied, then there exists sequences $\{x_{n}\}\uparrow\infty$ and $\{y_{n}\}\uparrow\infty$, such that
	\begin{equation*}		
		\lim_{n\rightarrow\infty}\sup_{x_{n}\leq  x\leq y_{n}}\left|\frac{P(S_{n}>x)}{nx^{-\delta/(1+\nu)}L(x)}-1\right|=0,
	\end{equation*}
	where $L(x)$ is a slowly varying function determined by the generating function of $\xi$ and $\eta$, and one can choose $x_{n}=n^{\frac{1+\nu}{\delta}+\kappa_{1}}$, $y_{n}=n^{\frac{1+\nu}{\nu}-\kappa_{2}}$ for any $\kappa_{1}$, $\kappa_{2}>0$ with $\kappa_{1}+\kappa_{2}<(1+\nu)(\nu^{-1}-\delta^{-1})$.
	\end{theorem}
	
\subsection{$\nu=0$: very heavy-tailed case}
When $\nu=0$, if $L_{1}(x)$ satisfies some requirements, for example $L_{1}(x)=(\log x)^{-a}$ with $a>0$, then 
	\begin{equation*}
		P(\xi>x)\sim cx^{-1}\tilde{L}_{1}(x),
	\end{equation*}
	where $\tilde{L}_{1}(x)$ is another slowly varying function, and it is called the {\it very heavy-tailed case}.

	Notice that by the criticality, $L_{1}(x)\rightarrow0$ as $x\rightarrow\infty$. So  we assume that for some $p>0$, as $x\rightarrow\infty$,
	\begin{equation}\label{conditionB}\tag{B}
		L_{2}(x)\sim pL_{1}(x),
	\end{equation}
	and thus $L_{2}(x)\rightarrow0$ as $x\rightarrow\infty$.
	
	For the distribution of $X$, \cite{Imomov} only consider the heavy-tailed case, here we also give the tail of $X$ in the very heavy-tailed case.

	\begin{theorem}\label{theoremX2}
		If $0=\nu<\delta<1$ and \eqref{conditionB} is satisfied, then as $x\rightarrow\infty$,
		\[P(X>x)\sim p\delta^{-1}(\Gamma(1-\delta))^{-1}x^{-\delta},\]
		where $\Gamma(\cdot)$ is the Gamma function.
	\end{theorem}
	
	\begin{remark}	
		As for the tail probability of the stationary distribution, previous research mainly focuses on the subcritical case. For example, Basrak et al. \cite{Basrak} proved that $X$ is also regularly varying  when $\xi$ or $\eta$ is regularly varying, and Foss and Miyazawa \cite{Foss} extended their results onto the strong subexponential case, which is more general. The critical case is more involved, Seneta \cite{Seneta67} considered an example that $f(x)=x+a\gamma^{-1}(1-x)^{\gamma}$	and $g(x)=e^{c(x-1)},$ where $1<\gamma<2$, $0<a<1$, $0<c<\infty$, then as $a\rightarrow0$, $c\rightarrow0$ such that $a/c\equiv$ constant, $X$ has a Poisson distribution ``in the limit". 
	\end{remark}

	We also obtain precise large deviation probability of partial sum $S_{n}$,
	\begin{theorem}\label{theorem2}	
	If $0=\nu<\delta<1$ and \eqref{conditionB} is satisfied, then there exists sequences $\{x_{n}\}\uparrow\infty$ and $\{y_{n}\}\uparrow\infty$, such that
	\begin{equation*}		
		\lim_{n\rightarrow\infty}\sup_{x_{n}\leq x\leq y_{n}}\left|\frac{P(S_{n}>x)}{nx^{-\delta}L(x)}-1\right|=0,
	\end{equation*}
	where $L(x)$ is a slowly varying function determined by the generating function of $\xi$ and $\eta$, and one can choose $x_{n}=n^{\frac{1}{\delta}+\kappa_{1}}$, $y_{n}=n^{\kappa_{2}}$ for any $\kappa_{1}$, $\kappa_{2}>0$ with $\kappa_{2}-\kappa_{1}>\delta^{-1}$.
	\end{theorem}
	
	\begin{remark}
	By comparing Theorem \ref{theorem1} and Theorem \ref{theorem2} with theorems in Guo and Hong \cite{Guo1} for subcritical case, we can intuitively see the difference of large deviations for partial sum between subcritical and critical branching processes. On the one hand, the interval of uniformity is different. In the critical case, the limit is uniform within an interval $(x_{n},y_{n})$ having lower bound and upper bound requirements; however, in the subcritical case, we need the lower bound only. On the other hand, in the subcritical case, the large deviations of $S_{n}$ are influenced by the offspring and immigration directly depending on which one is ``heavier", i.e., uniformly for $x\geq x_{n}$,  $P(S_{n}>x)\sim c_{1}nP(\eta>x)$ or $P(S_{n}>x)\sim c_{2}nP(\xi>x)$, with some constants $c_{1}$ and $c_{2}$ related to the mean of offspring and immigration. But in the critical case, both the offspring and immigration distributions are influential, i.e., $P(S_{n}>x)\sim nx^{-\frac{\delta}{1+\nu}}L(x)$, where $\delta$ is determined by the offspring and $\nu$ is determined by the immigration.
	\end{remark}
	
	\begin{remark}	
	Actually, the difference between the subcritical  and  critical case is caused by the comparison of the tail probabilities of the stationary distribution $X$ and that of  $\xi$ and $\eta$. In the subsequent proof, we will decompose $S_{n}$ into $S_{n,1}-S_{n,2}$ (see \eqref{decomSn}), where $S_{n,1}$ is determined by $\eta$ and $\xi$, and  $S_{n,2}$ is determined by $X$ and $\xi$. In the subcritical case that the immigration is heavier, the tail of $X$ has the same order as $\eta$, and thus $S_{n,2}$ is negligible for $x\in (x_{n},\infty)$. But in the critical case, by Theorem \ref{theoremX1} and \ref{theoremX2}, although the immigration is still heavier, the tail of $X$ is different to that of $\eta$, so we need to compare $x$ and $n$ to find proper interval $(x_{n},y_{n})$ such that $S_{n,2}$ is still negligible.
	\end{remark}		
	
	\begin{remark}	
	By comparing Theorem \ref{theorem1} and Theorem \ref{theorem2}, we can also find that the $x$-regions of uniformity in the latter can be longer, i.e., in the heavy-tailed case, the upper bound can approach but not reach $n^{\frac{1+\nu}{\nu}}$, but when the tail is very heavy, the upper bound can be as large as possible since $\kappa_{2}$ can be arbitrary large, and this is  closer to the subcritical case. In fact, under \eqref{conditionf}, it was shown in Slack \cite{Slack} that when $0<\nu<1$, the survival probability of the underlying branching process without immigration tends to zero with speed $n^{-1/\nu}$, roughly. But when $\nu=0$, with the condition $L_{1}(x)=(\log x)^{-a}$ with $a>0$, the speed is roughly $\exp\{-n^{\frac{1}{1+a}}\}$ (see Andreas and Sagitov \cite{Andreas}), which is more quickly and closer to the subcritical case.
	\end{remark}		
	
	The article is organized as follows. In Section 2 we give some regular variations of the considered process, including the tail behavior of the underlying branching process without immigration and the stationary distribution (Theorem \ref{theoremX2}). With the tail probability of $X$ in hand, in Section 3, we can obtain precise large deviation probabilities of partial sum $S_{n}$ (Theorem \ref{theorem1} and  \ref{theorem2}).

\section{Regular variation of the critical process}
\subsection{The underlying process}
	Let $\{Z_{n}\}$ be the underlying critical branching process (without immigration), which is defined by $Z_{0}=1$ and
	\begin{equation*}
		Z_{n}=\sum_{i=1}^{Z_{n-1}}\xi_{n,i}, \quad n\in\mathbb{N_{+}},
	\end{equation*}
	where $\{\xi_{n}\}$ is the same as in \eqref{defGWI1} and having the generating function \eqref{conditionf}.	
	
	Let
	\begin{equation}\label{defT}
		T:=\sum_{i=0}^{\infty}Z_{i}
	\end{equation}
	be the total population of $\{Z_{n}\}$. Denote the generating function of $T$ by $h(x):=Ex^{T}$,
	then it is well known that $h(x)$ satisfies the equation
	\[h(x)=xf(h(x)).\]
	
	Taking (\ref{conditionf}) into the last equation, we have, for $0\leq\nu<1$,
	\[h(x)=xh(x)+x(1-h(x))^{1+\nu}L_{1}((1-h(x))^{-1}).\]

	Thus as $x\uparrow1$, noticing that $h(1)=1$, and using Theorem 1.5.5 in Seneta \cite {Seneta76},
	\begin{equation}\label{eqh}
		1-h(x)\sim (1-x)^{1/(1+\nu)}L_{3}((1-x)^{-1}),
	\end{equation}
	where $L_{3}(x)$ is also a slowly varying function.		
	
	\begin{remark}
		{\rm 
		If $\nu\in(0,1)$, then by Tauberian Theorem, we get $T$ is regularly varying, i.e.,
	\begin{equation}\label{eqT}
		P(T>x)\sim cx^{-1/(1+\nu)}L_{3}(x),
	\end{equation}
	which was also proved by Lemma 6 in Vatutin et al.  \cite{Vatutin}.
		}
	\end{remark}

\subsection{The stationary distribution}	

	Define $P_{n}(x)$ as the generating function of $X_{n}$, i.e.,
	\begin{equation}\label{eqPn}
		P_{n}(x):=Ex^{X_{n}}=\prod_{k=0}^{n-1}g(f_{k}(x)),
	\end{equation}	
	where $f_{n}(x):=Ex^{Z_{n}}$ satisfies $f_{n}(x)=f(f_{n-1}(x))$. Then under \eqref{conditionf} and \eqref{conditiong}, as $n\rightarrow\infty$, $P_{n}(x)$ converges to the generating function
	\begin{equation}\label{eqP}
		P(x):=Ex^{X}=\prod_{k=0}^{\infty}g(f_{k}(x)).
	\end{equation}
	
	The next result gives the speed of the convergence of $P_{n}(x)$ to $P(x)$, which is an important tool when analyzing the tail of $X$. Result for $0<\nu<1$ is proved by \cite{Imomov}, and here we complete the conclusion for $\nu=0$ .
	
	\begin{lemma}\label{lemma}
		If $0=\nu<\delta<1$ and \eqref{conditionB} is satisfied, then for all $x\in[0,1)$, as $n\rightarrow\infty$,
		\begin{equation*}
			P_{n}(x)=P(x)\left[1+p\delta^{-1}\left(1-f_{n}(x)\right)^{\delta}\left(1+o(1)\right)\right].
		\end{equation*}		
	\end{lemma}
	\begin{proof}
		For the simplicity of notations, define $R_{n}(x):=1-f_{n}(x)$ and 
		$\Lambda(x):=L_{1}\left(x^{-1}\right),$
		then \eqref{conditionf} can be rewritten as
		\begin{equation}\label{eq21}
			R_{n+1}(x)=R_{n}(x)-R_{n}(x)\Lambda(R_{n}(x)).
		\end{equation}
		 
		Combining equation (4.14) - (4.17) in \cite{Imomov}, we have
		\begin{equation*}
			P_{n}(x)=P(x)(1+\Sigma_{n}(x)(1+o(1))),
		\end{equation*}
		where $\Sigma_{n}(x):=\sum_{k=n}^{\infty}[1-g(f_{k}(x))].$
		
		Slack \cite{Slack} has shown that when the process is critical, then as $n\rightarrow\infty$,
		\[U_{n}(x):=\frac{f_{n}(x)-f_{n}(0)}{f_{n}(0)-f_{n-1}(0)}\rightarrow U(x),\]
		where the limit function $U(x)$ is the generating function of the stationary measure of $\{Z_{n}\}$, satisfying $U(f(x))=U(x)+1$. By Lemma 2 in \cite{Imomov},
		\[U'(x)=\frac{\psi(x)}{(1-x)\Lambda(1-x)},\]
		where the function $\psi(x)$ satisfies $f'(x)\leq\psi(x)\leq1$ for $x\in[0,1)$ and $\psi(1-)=1$.
		
		Write $V(\cdot)$ as the inverse function of $U(1-x)$, then 
		\[1-f_{k}(x)=V(k+U(x)),\]
		\[1-g(f_{k}(x))=1-g[1-V(k+U(x))].\]
		
		By setting $1-y:=V(U(x)+t)$, we have $U(y)=U(x)+t$ and
		\begin{equation*}
			\begin{aligned}
				\int_{n}^{n+1}[1-g(1-V(U(x)+t))]dt
				=\int_{f_{n}(x)}^{f_{n+1}(x)}[1-g(y)]U'(y)dy
				=\int_{f_{n}(x)}^{f_{n+1}(x)}\psi(y)\frac{\mathscr{L}(1/1-y)}{(1-y)^{1-\delta}}dy,
			\end{aligned}
		\end{equation*}
		 where $\mathscr{L}(x):=L_{2}(x)\cdot L_{1}^{-1}(x)$ is also a slowly varying function and tends to $p$ as $x\rightarrow\infty$, by our assumption  \eqref{conditionB}. Then we have
		\[\int_{n}^{\infty}[1-g(1-V(U(x)+t))]dt\leq\Sigma_{n}(x)\leq\int_{n-1}^{\infty}[1-g(1-V(U(x)+t))]dt,\]
		i.e.,
		\begin{equation}\label{eqSigma}
		\Sigma_{n}(x)=\int_{f_{n}(x)}^{1}\psi(y)\frac{\mathscr{L}(1/1-y)}{(1-y)^{1-\delta}}+\sigma_{n}(x),
		\end{equation}
		where \[0\leq\sigma_{n}(x)\leq\int_{f_{n-1}(x)}^{f_{n}(x)}\psi(y)\frac{\mathscr{L}(1/1-y)}{(1-y)^{1-\delta}}dy.\]

		As for the first part of \eqref{eqSigma}, using mean value theorem and Karamata Theorem, for some $\theta=\theta(x)\in(0,1)$, as $n\rightarrow\infty$,
		\begin{equation*}
			\begin{aligned}
			\int_{f_{n}(x)}^{1}\psi(y)\frac{\mathscr{L}(1/1-y)}{(1-y)^{1-\delta}}dy
			&=\psi(1-\theta R_{n}(x)) \int_{f_{n}(x)}^{1}\frac{\mathscr{L}(1/1-y)}{(1-y)^{1-\delta}}dy\\
			&=\int_{R_{n}^{-1}(x)}^{\infty}y^{-(\delta+1)}\mathscr{L}(y)dy\cdot (1+o(1))\\
			&=p\delta^{-1}[R_{n}(x)]^{\delta}\cdot(1+o(1)).
			\end{aligned}
		\end{equation*}

	    As for the second part of \eqref{eqSigma}, for some $\theta=\theta(x)\in(0,1)$, $\phi=\phi(x)\in(0,1)$, define
	    \[\phi_{n}(x)=f_{n}(x)-\phi R_{n-1}(x)\Lambda(R_{n-1}(x)),\]
	    \[\theta_{n}(x)=R_{n}(s)+\theta R_{n-1}(x)\Lambda(R_{n-1}(x)),\]
	    then  by \eqref{eq21} and $\Lambda(x)\rightarrow0$ as $x\rightarrow0$,
		\begin{equation*}
			\begin{aligned}
			\int_{f_{n-1}(x)}^{f_{n}(x)}\psi(y)\frac{\mathscr{L}(1/1-y)}{(1-y)^{1-\delta}}dy
	    	&=\psi(\phi_{n}(x))\mathscr{L}(1/\theta_{n}(x))\int_{f_{n-1}(x)}^{f_{n}(x)}\frac{1}{(1-y)^{1-\delta}}dy\\
			&= p\int_{R_{n-1}^{-1}(x)}^{R_{n}^{-1}(x)}y^{-(\delta+1)}dy \cdot(1+o(1))\\
			&=p\delta^{-1}\left\{[R_{n-1}(x)]^{\delta}-[R_{n}(x)]^{\delta}\right\}\cdot (1+o(1))\\
			&=o([R_{n}(x)]^{\delta}),
			\end{aligned}
		\end{equation*}
		thus the lemma follows.
	\end{proof}
	
	Then we can prove that the stationary distribution $X$ for $\nu=0$ is also regularly varying, by a different method as that for $0<\nu<1$ in \cite{Imomov}.

	\begin{proof}[Proof of Theorem \ref{theoremX2}]
	Given $s\in(0,1)$, since for critical branching process $\{Z_{n}\}$, $f_{n}(0)\rightarrow1$, there exists $\lambda=\lambda(s)$, s.t.,
	\begin{equation}\label{eq31}
		f_{\lambda}(0)<s\leq f_{\lambda+1}(0).
	\end{equation}
	
	Notice that  $L_{1}(x)\rightarrow0$ when $\nu=0$,
	\begin{align*}
		\frac{1-f_{n+1}(0)}{1-f_{n}(0)}
		=\frac{1-f(f_{n}(0))}{1-f_{n}(0)}
		=1-L_{1}\left(\frac{1}{1-f_{n}(0)}\right)
		\rightarrow 1
	\end{align*}
	as $n\rightarrow\infty$, so as $s\uparrow1$,
	\begin{equation}\label{eq32}
		1-s\sim 1-f_{\lambda}(0). 
	\end{equation}

	From \eqref{eq31} we also have $P(f_{\lambda}(0))<P(s)\leq P(f_{\lambda+1}(0))$, then taking \eqref{eqP} into account, 
	\begin{equation*}
		1-\frac{P(0)}{P_{\lambda+1}(0)}\leq 1-P(s)< 1-\frac{P(0)}{P_{\lambda}(0)},
	\end{equation*}
	and thus by Lemma \ref{lemma},
	\begin{equation}\label{eq33}
		p\delta^{-1}(1-f_{\lambda+1}(0))^{\delta}(1+o(1))\leq 1-P(s)< p\delta^{-1}(1-f_{\lambda}(0))^{\delta}(1+o(1)).
	\end{equation}
	
	Then combining \eqref{eq32}--\eqref{eq33}, we get $1-P(s)\sim p\delta^{-1}(1-s)^{\delta}.$ And the theorem follows by Tauberian Theorem.	
\end{proof}

\section{Large deviation of $S_{n}$}

\subsection{Decomposition of $S_{n}$}

	To ease notation, we introduce the i.i.d. random operator $\theta_{n}(n\in\mathbb{N_{+}})$ as
	\[\theta_{n}\circ k=\sum_{i=1}^{k}\xi_{n,i}, \quad k\in\mathbb{N},\]
	where $\theta_{n}\circ 0=0$. And $\theta_{n}\circ(k_{1}+k_{2})\overset{\text{d}}{=}\theta_{n}^{(1)}\circ k_{1}+\theta_{n}^{(2)}\circ k_{2}$, where $\theta_{n}^{(1)}$ and $\theta_{n}^{(2)}$ on the right-hand side are independent with the same distribution as $\theta_{n}$. Then (\ref{defGWI1}) can be written as
	\begin{equation}\label{defGWI2}
		X_{n}=\theta_{n}\circ X_{n-1}+\eta_{n}, \quad n\in\mathbb{N_{+}}. 
	\end{equation}

	Define
	\begin{equation*}
		\Pi_{i,j}=\left\{
		\begin{aligned}
			&\theta_{j}\circ\theta_{j-1}\circ\cdots\circ\theta_{i}, &i\leq j,\\
			&1, &i>j.
		\end{aligned}
		\right.
	\end{equation*}
	
	Then
	$X_{n}=\sum_{i=1}^{n}\Pi_{i+1,n}\circ\eta_{i},$
	and we can write
	\begin{equation}\label{decomSn}
		\begin{aligned}
			S_{n}
			=\sum_{i=1}^{n}\sum_{m=i}^{\infty}\Pi_{i+1,m}\circ\eta_{i}
			-\sum_{i=1}^{n}\sum_{m=n+1}^{\infty}\Pi_{i+1,m}\circ\eta_{i}		
			:=S_{n,1}-S_{n,2}.
		\end{aligned}
	\end{equation}
	
	The first term of the right hand of (\ref{decomSn}) is
	\[
		S_{n,1}
		=\sum_{m=1}^{\infty}\Pi_{2,m}\circ\eta_{1}+\cdots+\sum_{m=n}^{\infty}\Pi_{n+1,m}\circ\eta_{n}
		:=Y_{1}^{(\infty)}+\cdots+Y_{n}^{(\infty)},
	\]
	where $\{Y_{i}^{(\infty)}\}_{i=1}^{n}$ are independent and have the same distribution as 
	\begin{equation}\label{decom1}
		Y^{(\infty)}\overset{\text{d}}{=}T^{(1)}+T^{(2)}+\cdots+T^{(\eta)},
	\end{equation}
	with $\{T^{(m)}\}_{m}$ being independent and having the same distribution as $T$ defined in (\ref{defT}).
	
	The second term  of the right hand in (\ref{decomSn}) is	
	\begin{align*}
		S_{n,2}
		:&=\sum_{i=1}^{n}(\sum_{m=n+1}^{\infty}\Pi_{n+1,m}\circ\Pi_{i+1,n}\circ\eta_{i})\\
		&=\sum_{m=n+1}^{\infty}\Pi_{n+1,m}\circ(\sum_{i=1}^{n}\Pi_{i+1,n}\circ\eta_{i})\\
		&=\sum_{m=n+1}^{\infty}\Pi_{n+2,m}\circ\theta_{n+1}\circ(\sum_{i=1}^{n}\Pi_{i+1,n}\circ\eta_{i}).
	\end{align*}
	
	So 
	\begin{equation}\label{decom2}
		S_{n,2}
		\overset{\text{d}}{=}T^{(1)}+T^{(2)}+\cdots+T^{(\theta\circ X_{n})},
	\end{equation}
	with $\{T^{(m)}\}_{m}$ being independent and having the same distribution as $T$ defined in (\ref{defT}).	

\subsection{Main idea of the proof}

	Observe that, for $\forall$ small $\varepsilon>0$,
	\begin{align*}
		&P(S_{n,1}>(1+\varepsilon)x)-P(S_{n,2}>\varepsilon x)\\
		\leq &P(S_{n,1}-S_{n,2}>x)\\
		\leq &P(S_{n,1}>(1-\varepsilon)x)+P(-S_{n,2}>\varepsilon x).
	\end{align*}

	We will show that $S_{n,2}$ does not contribute to the considered large deviations and the main order comes from $S_{n,1}$.
	
	For $S_{n,1}$, notice that it is the summation of i.i.d. regularly varying random variables, so we can use the known large deviation result to estimate.

	For $S_{n,2}$, we have, as $n\rightarrow\infty$, $X_{n}\rightarrow X$ increasingly in distribution. In fact, the conclusion comes from
		$X_{n}\overset{d}{=}\sum_{k=0}^{n-1}D_{n},$
		where $D_{0}:=\eta_{0}$ and $D_{n}:=\theta_{n}^{(n)}\circ\theta_{n-1}^{(n)}\circ\cdots\circ\theta_{1}^{(n)}\circ\eta_{n}$ $(n\geq1)$, $\{\eta_{n}\}_{n\geq0}$ are independent with the same distribution as $\eta$, $\{\theta_{i}^{(n)}\}_{n\geq1}$ are independent with the same distribution as $\theta_{i}$, and
		$X\overset{d}{=}\sum_{n=0}^{\infty}D_{n}.$
		Then we have $S_{n,2}$ converges to the limit
		\[S^{(\infty)}:=T^{(1)}+T^{(2)}+\cdots+T^{(\theta\circ X)}.\]
		By the distribution of $T$ and $X$, and the inheritable properties of regularly varying functions, we can get the tail of $S^{(\infty)}$ and prove that this part is negligible.

\subsection{Estimate of $S_{n,1}$}
	The tail behavior of $S_{n,1}$ is the same for $\nu\in(0,1)$ and $\nu=0$, since it is composed by $n$ i.i.d. random variables, each of which is regularly varying with index $\frac{\delta}{1+\nu}\in(0,1)$.
	
	\begin{proposition}\label{prop1Sn1}
	For $\nu\in[0,1)$ and $\nu<\delta<1$,
	\[\lim_{n\rightarrow\infty}\sup_{x\geq x_{n}}
	\left|\frac{P(S_{n,1}>x)}{nx^{-\delta/(1+\nu)}L_{4}(x)}-1\right|=0,\]
	where one can choose $x_{n}=n^{\frac{1+\nu}{\delta}+\kappa_{1}}$ for any $\kappa_{1}>0$ and $L_{4}(x)$ is a slowly varying function determined by the generating functions of $\xi$ and $\eta$.
	\end{proposition}
	\begin{proof}
	Combining \eqref{conditiong} and \eqref{eqh}, we have
	\[
			1-E(x^{Y^{(\infty)}})
			=1-g(h(x))
			\sim (1-x)^{\delta/(1+\nu)}\left(L_{3}\left(\frac{1}{1-x}\right)\right)^{\delta}L_{2}\left(\frac{1}{(1-x)^{1/(1+\nu)}L_{3}\left(\frac{1}{1-x}\right)}\right),		
		\]
	which means 
	\[P(Y^{(\infty)}>x)\sim x^{-\delta/(1+\nu)}L_{4}(x)\]
	as $x\rightarrow\infty$, where $L_{4}(x)$ is a slowly varying function determined by $L_{2}$ and $L_{3}$.
	
	Then $\forall$ $\varepsilon>0$, $\exists$ $x(\varepsilon)>0$, 
	\[\sup_{x\geq x(\varepsilon)}\left|\frac{P(Y^{(\infty)}>x)}{x^{-\delta/(1+\nu)}L_{4}(x)}-1\right|<\varepsilon.\]
	
	For any sequence $a_{n}\rightarrow\infty$, we can choose $N(\varepsilon)$ such that $\forall n>N$, $a_{n}>x(\varepsilon)$, and thus
	\[\sup_{x\geq a_{n}}\left|\frac{P(Y^{(\infty)}>x)}{x^{-\delta/(1+\nu)}L_{4}(x)}-1\right|\leq\sup_{x\geq x(\varepsilon)}\left|\frac{P(Y^{(\infty)}>x)}{x^{-\delta/(1+\nu)}L_{4}(x)}-1\right|<\varepsilon,\]
	so we get
	\begin{equation}\label{eq4A1}
		\lim_{n\rightarrow\infty}\sup_{x\geq a_{n}}\left|\frac{P(Y^{(\infty)}>x)}{x^{-\delta/(1+\nu)}L_{4}(x)}-1\right|=0.
	\end{equation}		
	
	Recall that $S_{n,1}=Y_{1}^{(\infty)}+\cdots+Y_{n}^{(\infty)}$
	is the summation of $n$ independent random variables having the same distribution as $Y^{(\infty)}$, using the large deviation probability result of i.i.d. regularly varying random variables, we have
	\begin{equation}\label{eq4A2}
		\lim_{n\rightarrow\infty}\sup_{x\geq x_{n}}
		\left|\frac{P(S_{n,1}>x)}{nP(Y^{(\infty)}>x)}-1\right|=0,
	\end{equation}
	where one can choose $x_{n}=n^{\frac{1+\nu}{\delta}+\kappa_{1}}$ for any $\kappa_{1}>0$ (for example, see A.V. Nagaev \cite{AV}, S.V. Nagaev   \cite{SV}, Cline and Hsing \cite{Cline}, or Theorem 1.2 of Buraczewski et al.  \cite{Buraczewski}).
	
	Combining equation (\ref{eq4A1}) and (\ref{eq4A2}),
	\begin{align*}
		0\leq &\varliminf_{n\rightarrow\infty}\sup_{x\geq x_{n}}
		\left|\frac{P(S_{n,1}>x)}{nx^{-\delta/(1+\nu)}L_{4}(x)}-1\right|\\
		\leq &\varlimsup_{n\rightarrow\infty}\sup_{x\geq x_{n}}
		\left|\frac{P(S_{n,1}>x)}{nx^{-\delta/(1+\nu)}L_{4}(x)}-1\right|\\
		\leq&\varlimsup_{n\rightarrow\infty}\sup_{x\geq x_{n}}\left|\frac{P(S_{n,1}>x)}{nP(Y^{(\infty)}>x)}-1\right|\cdot
		\sup_{x\geq x_{n}}\frac{P(Y^{(\infty)}>x)}{x^{-\delta/(1+\nu)}L_{4}(x)}\\
		&+\varlimsup_{n\rightarrow\infty}\sup_{x\geq x_{n}}\left|\frac{P(Y^{(\infty)}>x)}{x^{-\delta/(1+\nu)}L_{4}(x)}-1\right|=0,
	\end{align*}
	and the proposition follows. 
\end{proof}

\subsection{Estimate of $S_{n,2}$}

	Here, for the tail of $S_{n,2}$, the  heavy-tailed and very heavy-tailed cases exhibit different properties, so we summarize them in two propositions.
	
	\begin{proposition}\label{prop1Sn2}
	If $0<\nu<\delta<1$ and  \eqref{conditionA1}-\eqref{conditionA2} are satisfied, then
	\[\lim_{n\rightarrow\infty}\sup_{a_{n}\leq x\leq y_{n}}\frac{P(S_{n,2}>x)}{nx^{-\delta/(1+\nu)}L_{4}(x)}=0,\]
	where one can choose $y_{n}=n^{\frac{1+\nu}{\nu}-\kappa_{2}}$ for any $0<\kappa_{2}<\frac{1+\nu}{\nu}$, and any sequence $a_{n}\rightarrow\infty$, with $a_{n}\leq y_{n}$ for all $n\in\mathbb{N_{+}}$. $L_{4}(x)$ is the  slowly varying function defined in Proposition \ref{prop1Sn1}. 
	\end{proposition}
	\begin{proof}
	Since $X$ is regularly varying with index $\delta-\nu<1$ and $E\xi=1$, by the inheritable property of random sum (for example, see Fay et al.  \cite{Fay}, Robert et al.  \cite{Robert}, Barczy et al.  \cite{Barczy}, or the appendix of Guo and Hong \cite{Guo1}), we have, as $x\rightarrow\infty$,
	\[P(\theta\circ X>x)=P\left(\sum_{i=1}^{X}\xi_{i}>x\right)\sim P(X>x)\sim cx^{-(\delta-\nu)}.\]
		
	So using the same method as in the beginning of Proposition \ref{prop1Sn1}, we have, as $x\rightarrow\infty$,
	\begin{equation*}	
		P(S^{(\infty)}>x)\sim x^{-(\delta-\nu)/(1+\nu)}L_{5}(x),
	\end{equation*}
	where $L_{5}(x)$ is a slowly varying function.
		
	Define
	\[r_{1}:=\frac{\nu}{1+\nu}+\frac{\kappa_{2}\nu^{2}}{2(1+\nu)(1+\nu-\kappa_{2}\nu)}\]
	and
	\[r_{2}:=\frac{\kappa_{2}\nu^{2}}{2(1+\nu)(1+\nu-\kappa_{2}\nu)}.\]
	From $\kappa_{2}<(1+\nu)/\nu,$ we have $r_{2}>0$ and $0<(\frac{1+\nu}{\nu}-\kappa_{2})r_{1}<1$.
	
	Since as $n\rightarrow\infty$,
    $S_{n,2}$ converges to $S^{(\infty)}$ increasingly in distribution, using the properties of slowly varying functions, we get 
	\begin{align*}
		\sup_{a_{n}\leq x\leq y_{n}}\frac{P(S_{n,2}>x)}{nx^{-\delta/(1+\nu)}L_{4}(x)}
		\leq & \sup_{a_{n}\leq x\leq y_{n}}\frac{P(S^{(\infty)}>x)}{nx^{-\delta/(1+\nu)}L_{4}(x)}\\
		= & \sup_{a_{n}\leq x\leq y_{n}}\left(\frac{P(S^{(\infty)}>x)}{x^{-(\delta-\nu)/(1+\nu)}L_{5}(x)}\cdot \frac{x^{r_{1}}}{n}\cdot x^{-r_{2}}\frac{L_{5}(x)}{L_{4}(x)}\right)\\
		\leq &
		\frac{y_{n}^{r_{1}}}{n}
		\cdot
		\sup_{x\geq a_{n}}\frac{P(S^{(\infty)}>x)}{x^{-(\delta-\nu)/(1+\nu)}L_{5}(x)}
		\cdot
		\sup_{x\geq a_{n}}x^{-{r_{2}}}\frac{L_{5}(x)}{L_{4}(x)}\\
		\rightarrow & 0
	\end{align*}
	as $n\rightarrow\infty$.
	\end{proof}

	\begin{proposition}\label{prop2Sn2}
	If $0=\nu<\delta<1$ and  \eqref{conditionB} is satisfied, then
	\[\lim_{n\rightarrow\infty}\sup_{a_{n}\leq x\leq y_{n}}\frac{P(S_{n,2}>x)}{nx^{-\delta}L_{4}(x)}=0,\]
	where one can choose $y_{n}=n^{\kappa_{2}}$ for any $\kappa_{2}>0$, and any sequence $a_{n}\rightarrow\infty$ with $a_{n}\leq y_{n}$ for all $n\in\mathbb{N_{+}}$. $L_{4}(x)$ is the  slowly varying function defined in Proposition \ref{prop1Sn1}. 
	\end{proposition}
	\begin{proof}
	Same as in Proposition \ref{prop1Sn2}, we have, as $x\rightarrow\infty$,
	\begin{equation*}	
		P(S^{(\infty)}>x)\sim x^{-\delta}L_{5}(x),
	\end{equation*}
	where $L_{5}(x)$ is a slowly varying function. Thus using the properties of slowly varying functions,
		\begin{align*}
		\sup_{a_{n}\leq x\leq y_{n}}\frac{P(S_{n,2}>x)}{nx^{-\delta}L_{4}(x)}
		\leq & \sup_{a_{n}\leq x\leq y_{n}}\frac{P(S^{(\infty)}>x)}{nx^{-\delta}L_{4}(x)}\\
		= & \sup_{a_{n}\leq x\leq y_{n}}\left(\frac{P(S^{(\infty)}>x)}{x^{-\delta}L_{5}(x)}\cdot \frac{x^{\frac{1}{2\kappa_{2}}}}{n}\cdot x^{-\frac{1}{2\kappa_{2}}}\frac{L_{5}(x)}{L_{4}(x)}\right)\\
		\leq &
		\frac{y_{n}^{\frac{1}{2\kappa_{2}}}}{n}
		\cdot
		\sup_{x\geq a_{n}}\frac{P(S^{(\infty)}>x)}{x^{-\delta}L_{5}(x)}
		\cdot
		\sup_{x\geq a_{n}}x^{-\frac{1}{2\kappa_{2}}}\frac{L_{5}(x)}{L_{4}(x)}\\
		\rightarrow & 0
	\end{align*}
	as $n\rightarrow\infty$.
	\end{proof}
	
	\begin{remark}
		{\rm
		Actually it is shown in the last inequality that we just need $y_{n}=O(n^{\gamma})$ for some $\gamma>0$, since the second term on the right hand tends to $1$ and the third term tends to $0$.
		}
	\end{remark}
	
\subsection{Proof of the main result}

\begin{proof}[Proof of Theorem \ref{theorem1} and \ref{theorem2}]

Let $\{x_{n}\}$ and $\{y_{n}\}$ be sequences satisfy the conditions in Proposition \ref{prop1Sn1}--\ref{prop2Sn2}. By the slow variation of $L_{4}(x)$, we have, for $\forall\varepsilon>0$,

	\[\lim_{n\rightarrow\infty}\sup_{x\geq x_{n}}\left|\frac{P(S_{n,1}>(1+\varepsilon)x)}{nx^{-\delta/(1+\nu)}L_{4}(x)}-1\right|
	=\lim_{n\rightarrow\infty}\sup_{x\geq x_{n}}\left|\frac{P(S_{n,1}>(1-\varepsilon)x)}{nx^{-\delta/(1+\nu)}L_{4}(x)}-1\right|=0\]
 and
	\[\lim_{n\rightarrow\infty}\sup_{x_{n}\leq x\leq y_{n}}\frac{P(S_{n,2}>\varepsilon x)}{nx^{-\delta/(1+\nu)}L_{4}(x)}=0.\]

	Then  we can conclude that for $0\leq \nu<\delta<1$,
	\begin{align*}
		0
		\leq&\varliminf_{n\rightarrow\infty}\sup_{x_{n}\leq  x\leq y_{n}}\left|\frac{P(S_{n}>x)}{nx^{-\delta/(1+\nu)}L(x)}-1\right|\\
		\leq&\varlimsup_{n\rightarrow\infty}\sup_{x_{n}\leq  x\leq y_{n}}\left|\frac{P(S_{n}>x)}{nx^{-\delta/(1+\nu)}L(x)}-1\right|\\
		\leq&\max
		\left\{\varlimsup_{\varepsilon\rightarrow0}\varlimsup_{n\rightarrow\infty}\sup_{x_{n}\leq  x\leq y_{n}}
		\left(\left|\frac{P(S_{n,1}>(1+\varepsilon)x)}{nx^{-\delta/(1+\nu)}L(x)}-1\right|
		+\frac{P(S_{n,2}>\varepsilon x)}{nx^{-\delta/(1+\nu)}L(x)}\right), \right.\\
		&\left.\qquad\;\;
		\varlimsup_{\varepsilon\rightarrow0}\varlimsup_{n\rightarrow\infty}\sup_{x_{n}\leq  x\leq y_{n}}
		\left|\frac{P(S_{n,1}>(1-\varepsilon)x)}{nx^{-\delta/(1+\nu)}L(x)}-1\right|
		\right\}\\
		\leq&\max
		\left\{
		\lim_{\varepsilon\rightarrow0}\lim_{n\rightarrow\infty}\sup_{x\geq 	x_{n}}\left|\frac{P(S_{n,1}>(1+\varepsilon)x)}{nx^{-\delta/(1+\nu)}L(x)}-1\right|
		+\lim_{\varepsilon\rightarrow0}\lim_{n\rightarrow\infty}\sup_{x_{n}\leq  x\leq 	y_{n}}\frac{P(S_{n,2}>\varepsilon x)}{nx^{-\delta/(1+\nu)}L(x)}, \right.\\
		&\left.\qquad\;\;
		\lim_{\varepsilon\rightarrow0}\lim_{n\rightarrow\infty}\sup_{x\geq 	x_{n}}\left|\frac{P(S_{n,1}>(1-\varepsilon)x)}{nx^{-\delta/(1+\nu)}L(x)}-1\right|
		\right\}\\
		=& 0,	
	\end{align*}
	with the notation that $L(x):=L_{4}(x)$, which completes the proof.
\end{proof}

\end{document}